\documentclass[12pt]{article}
\usepackage{amssymb,amsmath,enumerate, amsthm, graphicx}

\newtheorem{theorem}{Theorem}[section]
\newtheorem{lemma}[theorem]{Lemma}

\newtheorem{definition}[theorem]{Definition}

\newtheorem{conjecture}[theorem]{Conjecture}
\newtheorem{claim}[theorem]{Claim}

 \newcommand{\eps}{\varepsilon}                       
      \renewcommand{\epsilon}{\varepsilon}

\title{Spanning Trees in Graphs
of High Minimum Degree with
a Universal Vertex II:\\ A Tight Result} 
\author{
Bruce Reed\footnote{University of Victoria. Research supported by NSERC.} 
\quad
Maya Stein\footnote{University of Chile, Research supported by ANID Regular Grant 1221905, by FAPESP-ANID Investigaci\'on Conjunta grant 2019/13364-7, and by ANID PIA CMM FB210005.}
}
\date{}
\begin{document}
\maketitle

\begin{abstract}
We prove that, if $m$ is  
sufficiently large, every graph on $m+1$ vertices that has a universal vertex and minimum degree at least $\lfloor \frac{2m}{3} \rfloor$ contains each tree $T$ with $m$ edges as a subgraph.  
Our result confirms, for large $m$, an important special case of a conjecture by Havet, Reed, Stein, and Wood. 

The present paper builds on the results of a companion paper in which we proved the statement for all trees having a vertex that  is adjacent to many leaves.
\end{abstract}

\section{Introduction}
%This is the second in a series of two papers proving a result relating the minimum and the maximum degree of a graph to the occurence of all spanning trees as subgraphs. 

It is easy to see that any graph of minimum degree at least~$m$ contains a copy of each tree with $m$ edges, and that this bound is sharp.
Variants replacing the minimum degree condition with another degree condition have also been proposed. The average degree is used in the well-known
 Erd\H os--S\' os conjecture  (see~\cite{BPS2, rohzon} for recent results), and the median degree is used in the Loebl-Koml\'os-S\'os conjecture, which was
 approximately solved  in~\cite{cite:LKS-cut0, cite:LKS-cut1, cite:LKS-cut2, cite:LKS-cut3}. These variants are strengthenings of the observation at the beginning of the paragraph.

If, however, one wishes to strengthen the observation by simply  weakening the imposed bound on the minimum degree of the host graph, the problem becomes impossible. For this, it suffices to consider the disjoint union of complete graphs of order $m$. This graph has minimum degree $m-1$ and contains no tree with $m$ edges. 

But, if we are restricting our attention to spanning trees,
it is still possible to embed bounded degree  trees using a weaker minimum degree condition.  Koml\'os, Sark\"ozy and Szemer\'edi showed in~\cite{KSS} that for every $\delta >0$,  every large enough $(m+1)$-vertex graph of minimum degree at least $(1+\delta)\frac m2$ contains each tree with $m$ edges whose maximum degree is bounded by $O(\frac{n}{\log n})$. 
From the example given above, it is clear, though, that an analogue of the result from~\cite{KSS} could not be true for trees that are much smaller than the host graph, even if we would require a minimum degree of just below the size of the tree we are looking for.  

So, it seems natural to seek an additional condition to impose on the host graph to make a statement in this direction come true. A condition on the maximum degree is an obvious candidate, since we may have to embed a star with $m$ edges. The following conjecture in this respect has been put forward recently. 

\begin{conjecture}[Havet, Reed, Stein, and Wood~\cite{HRSW}]%$\!\!${\rm\bf\cite{HRSW}}
\label{conj1}
Let $m\in \mathbb N$. 
If a graph has maximum  degree at least $m$ and minimum degree at least $\lfloor \frac{2m}{ 3}  \rfloor$ then it contains every tree
with $m$ edges as a subgraph.
\end{conjecture}

We remark that  if the minimum degree condition is replaced by the much stronger bound $(1-\gamma)m$, for a tiny constant $\gamma$, a result along the lines of  Conjecture \ref{conj1} is true~\cite{HRSW}. The conjecture also holds if the maximum degree condition is replaced by a large function in~$m$~\cite{HRSW}. 

Furthermore, an approximate version of Conjecture \ref{conj1} holds for bounded degree trees and dense host graphs~\cite{BPS}. Such an approximate version even holds for a generalised form of Conjecture \ref{conj1}, where the bound on the minimum degree is allowed to be any value between $\frac m2$ and $\frac{2m}3$, with the maximum degree obeying a corresponding bound between $2m$ and $m$ (see~\cite{BPS3} for details).

As further evidence for Conjecture \ref{conj1}, we prove that it holds when the graph has $m+1$ vertices, if $m$ is large enough. That is, we show the conjecture for the case when we are looking for a spanning tree in a large graph. 

\begin{theorem} 
\label{maint}  There is an $m_0\in\mathbb N$ such that for every $m \ge m_0$ 
every graph on $m+1$ vertices that has minimum degree at least $\lfloor \frac{2m}{ 3}  \rfloor$ and a universal vertex contains every tree
$T$ with $m$ edges as a subgraph.
\end{theorem}

Clearly, our theorem can also be understood as a variant of the result by Koml\'os, Sark\"ozy and Szemer\'edi mentioned above.

The proof of Theorem~\ref{maint} follows quickly from a result obtained
in the companion paper~\cite{brucemaya1}, and  a second result, Lemma~\ref{theotherlemma}, which will be proved in the present paper. We present the two lemmas and give the short proof of Theorem~\ref{maint} in the next section, deferring the proof of Lemma~\ref{theotherlemma} to the subsequent sections.

\section{Proof of Theorem~\ref{maint}}

In the companion paper~\cite{brucemaya1}, we showed the following lemma.
\begin{lemma}
\label{realhighleafdegreecase}$\!\!${\rm\bf\cite[Lemma 1.3]{brucemaya1}}
For every $\delta>0$, there is an $m_\delta$ such that for any $m \ge m_\delta$  the following holds for every graph $G$ on $m+1$ vertices that has minimum degree at least $\lfloor \frac{2m}{ 3}  \rfloor$ and a universal vertex. \\ If $T$ is a tree with $m$ edges, and some vertex of $T$ is  adjacent 
to at least $\delta m$ leaves, then~$T$ embeds in $G$.
\end{lemma}

Lemma~\ref{realhighleafdegreecase} covers the proof of our main result  for all trees which have a vertex with many leaves, namely at least $\delta m$ leaves, for some fixed $\delta$, but is of no help for trees which have no such vertex.
 This latter case is covered by the next lemma which will be proved in the present paper.
 
\begin{lemma}\label{theotherlemma}
There are $m_1\in \mathbb N$, and  $\delta>0$ such that the following holds for every $m\geq m_1$, and every graph $G$ on $m+1$ vertices that has minimum degree at least $\lfloor \frac{2m}{ 3}  \rfloor$ and a universal vertex.\\
If $T$ is a tree with $m$ edges such that no vertex of $T$ is adjacent to more than $\delta m$ leaves, then $T$ embeds in $G$.
\end{lemma}

The proof of Lemma~\ref{theotherlemma} is given in the next section,  Section~\ref{proofofthel}. It will rely on four auxiliary lemmas, Lemmas~\ref{cutthetree}, \ref{bruce}, \ref{maya1}, and~\ref{maya2}, of which one is proved right away in Section~\ref{proofofthel}, one is from~\cite{brucemaya1}, and the remaining two will be proved in later sections of the present paper.

With Lemma~\ref{realhighleafdegreecase} and  Lemma~\ref{theotherlemma} at hand, the proof of our main result, Theorem~\ref{maint}, is straightforward.

\begin{proof}[Proof of Theorem~\ref{maint}]
We choose our output $m_0$ for Theorem~\ref{maint} by taking the maximum value of $m_1$ and $m_\delta$, where $m_1$ and $\delta$ are given by Lemma~\ref{theotherlemma}, and $m_\delta$ is given for input $\delta$ by Lemma~\ref{realhighleafdegreecase}.
Given now $T$ and $G$ as in the theorem, Lemma~\ref{theotherlemma} covers the case that  $T$ has no vertex adjacent to more than $\delta m$ leaves, and Lemma~\ref{realhighleafdegreecase} covers the remaining case. 
\end{proof}

\section{The proof of Lemma~\ref{theotherlemma}}\label{proofofthel}

We start by giving a quick overview of the proof of Lemma~\ref{theotherlemma} in Section~\ref{idea}.  As mentioned earlier, we formally organise the proof of Lemma~\ref{theotherlemma} by splitting it up into four auxiliary lemmas, namely Lemma~\ref{cutthetree}, Lemma~\ref{bruce}, Lemma~\ref{maya1}, and Lemma~\ref{maya2}. 
These four auxiliary lemmas  will be stated in Section~\ref{4lemmas}. Section~\ref{proofofthelemma} then contains the proof of Lemma~\ref{theotherlemma}, under the assumption that the four auxiliary lemmas hold. 

The easy proof of Lemma~\ref{cutthetree} is given in Section~\ref{gammalemma}. 
Lemma~\ref{maya1} was proved in~\cite{brucemaya1}. So, at the end of this section, there will be only two lemmas, Lemma~\ref{bruce} and Lemma~\ref{maya2}, left  to prove. 
 In  Sections~\ref{br1} and~\ref{br2}, we  state and prove two new lemmas, Lemma~\ref{easybruce} and Lemma~\ref{hardbruce},  which together imply Lemma~\ref{bruce}. 
The last section of the paper, Section~\ref{ma2}, is devoted to the proof  of  Lemma~\ref{maya2}.

\subsection{Idea of the proof of Lemma~\ref{theotherlemma}}\label{idea}

The idea of the proof is to first reserve a random set $S\subseteq V(G)$ for later use. Then, we embed into $G-S$ a very small subtree $T^*$ of the tree $T$ we wish to embed. Actually, we will only embed  $T^*-L$, having chosen a subset $L\subseteq V(T^*)$ of some low degree vertices (either leaves or vertices of degree 2). The vertices from $L$ will be left out of the embedding for now, as they will only be embedded at the very end.
 
The set $L$ is slightly larger than the set $S$. This gives us some free space when we embed $T-T^*$, which will be useful. In fact, this freedom makes it possible for us to use a  lemma from~\cite{brucemaya1} (stated as Lemma~\ref{maya1} in the present paper) for embedding $T-T^*$, unless the graph $G$ has a very special structure, in which case an ad-hoc embedding is provided by Lemma~\ref{maya2}. After this, there is a small leftover set of vertices of~$G$, which, together with the set $S$, serves for embedding the  vertices from
 $L$, by using an absorption argument.

\subsection{Four auxiliary lemmas}\label{4lemmas}

In the present section, we present our four  auxiliary lemmas, Lemma~\ref{cutthetree}, Lemma~\ref{bruce}, Lemma~\ref{maya1}, and Lemma~\ref{maya2}.

We start with the simplest of our  lemmas,  Lemma~\ref{cutthetree}. This lemma enables us to find a convenient subtree $T^*$ of a tree~$T$.
We need a quick definition before we give the lemma.

\begin{definition}[$\gamma$-nice subtree, type 1, type 2]\label{gamma-nice}
Let $T$ be a tree with $m$ edges.
Call  a subtree~$T^*$ of  $T$ with root $t^*$ a {\it $\gamma$-nice subtree} if
\begin{enumerate}[(i)]
\item $|V(T^*)|\le\gamma m$; and \label{Tstarsmall}
\item every component of $T-T^*$  is adjacent to $t^*$.\label{compo}
\end{enumerate}
Consider the following additional conditions:
\begin{enumerate}[(1)]
\item $T^*$ contains at least $\lceil\frac{\gamma m}{20}\rceil$ disjoint paths of length $5$ and all vertices on these paths have degree at most $2$ in~$T$. \label{manypathsinT^*2}
\item   $T^*$ contains at least $\lceil\frac{\gamma m}{40}\rceil$ leaves from $T$.\label{manyleavesinT^*22}
\end{enumerate}
If the former condition holds, we say $T^*$ is of type 1, and if the latter condition holds, we say $T^*$ is of type 2.
\end{definition}

We are now ready to state the lemma that finds a $\gamma$-nice subtree of one of the two types.

\begin{lemma}\label{cutthetree}
For all $0<\gamma\le 1$, any tree with  $m\ge \frac {200}\gamma$ edges has a $\gamma$-nice subtree of type 1 or of type 2.
\end{lemma}

The proof of Lemma~\ref{cutthetree} is straightforward, but we prefer to leave it to the end of the present section, namely to Subsection~\ref{gammalemma}, 
 in order to be able to first  focus on the proof of the main result.
 
Next, we exhibit a lemma that will enable us to transfer the embedding problem of the tree to an embedding problem of almost all of the tree, under the condition that we already embed a small part of it, i.e.~a $\gamma$-nice subtree, beforehand.

For convenience,  let us use the following notation.
\begin{definition}[$m$-good graph]
Let $m\in\mathbb N$. Call 
 a graph  {\it $m$-good} if it has $m+1$ vertices, minimum degree at least $\lfloor \frac{2m}{ 3}  \rfloor$ and a universal vertex. 
\end{definition}

\begin{lemma}
\label{bruce}
There is an $m_0\in \mathbb N$ such that the following holds for all $m \ge m_0$, and all $\gamma $ with  $\frac 2{10^{7}}\le\gamma<\frac 1{30}$. 

Let $G$ be an $m$-good graph, with universal vertex $w$. Let $T$ be a tree with $m$ edges, such that no vertex of $T$ is  adjacent to more than $\frac m{10^{23}}$ leaves. Let~$T^*$ be a $\gamma$-nice
 subtree of $T$, of type 1 or 2, rooted at vertex $t^*$. \\
 Then there are  sets $L\subseteq V(T^*)\setminus\{t^*\}$ and $S\subseteq V(G)$  satisfying
$$|S| \leq |L|-\lceil(\frac{\gamma}2)^4 m\rceil.$$ Furthermore, for any $w'\in V(G)-S$, with $w'\neq w$, there is 
 an embedding of $T^*-L$ into  $G-S$, with~$t^*$ embedded in $w'$, such that the following holds. Any embedding of $T-L$ into $G-S$ extending our embedding of $T^*-L$  can be extended to an embedding of all of $T$ into $G$.
\end{lemma}

As mentioned earlier, later on  we shall split Lemma~\ref{bruce} into two lemmas,
Lemma  \ref{easybruce} and  Lemma \ref{hardbruce}, depending on the type of the $\gamma$-nice subtree. We will state and prove Lemma \ref{easybruce}  in Section~\ref{br1}, and state and prove  Lemma \ref{hardbruce} in Section~\ref{br2}. 
Together, Lemmas  \ref{easybruce} and   \ref{hardbruce} imply Lemma~\ref{bruce}.

In order to state the remaining two of our four auxiliary lemmas, we need a simple definition. This definition describes the extremal case, where the graph $G$ has a very specific structure (and therefore, the approach from the companion paper~\cite{brucemaya1} does not work).

\begin{definition}\label{gammaspecial}
Let $\gamma>0$. We say a graph $G$ on $m+1$ vertices is $\gamma$-special if $V(G)$ consists of three mutually disjoint sets $X_1,X_2,X_3$ such that
\begin{itemize}
\item  $\frac{m}{3}-3{\gamma} m\leq |X_i|\leq\frac{m}{3}+3{\gamma} m$ for each $i=1,2,3$; and
\item there are at most $\gamma^{10} |X_1|\cdot|X_2|$ edges between $X_1$ and $X_2$.
%; and 
%\item every vertex of $X_3$ sees all but $10^{6} \sqrt{\gamma} m$ vertices of $X_1 \cup X_2$
\end{itemize}
\end{definition}

%Observe that the trees $T^*$ from Lemmas~\ref{easybruce} and~\ref{hardbruce} are $\gamma$-nice.

\medskip

The following lemma, which excludes the extremal situation, was proved  in the companion paper~\cite{brucemaya1}.

\begin{lemma}$\!\!${\rm\bf\cite[Lemma 7.3]{brucemaya1}}
\label{maya1}
For all $\gamma<\frac 1{10^6}$ 
 there are $m_0\in\mathbb N$ and  $\lambda>0$ such that the following holds for all $m\geq m_0$. \\ Let $G$ be an $m$-good  graph,  which is not  $\gamma$-special.  Let~$T$ be a tree with $m$ edges such that $T\not\subseteq G$ and no vertex in $T$ is adjacent to more than $\lambda m$ leaves.  Let~$T^*$ be a $\gamma$-nice subtree of~$T$, with root $t^*$,  let $L\subseteq V(T^*)\setminus\{t^*\}$, and let $S\subseteq V(G)$ such that $|S|\le |L|-\lceil(\frac\gamma 2)^4 m\rceil$.\\
 Assume that for any  $W\subseteq V(G)-S$ with $|W|\geq \gamma m$, there is an embedding $\phi_W$ of $T^*-Y$ into $G- S$, with $t^*$ embedded in $W$. Then there is a set $W\subseteq V(G)-S$ with $|W|\geq \gamma m$, and an embedding of $T-Y$  into $G-S$ that extends $\phi_W$. 
%If there is an embedding of $T^*-L$ into $G- S$, then
% there is an embedding of $T-L$  into $G-S$ extending the 
%embedding of $T^*-L$. 
\end{lemma}

Our last auxiliary lemma deals with the extremal case  described in Definition~\ref{gammaspecial}. 

\begin{lemma}
\label{maya2}
There are $m_0\in\mathbb N$, $\beta\le\frac 1{10^{10}}$,   and $\gamma_0, \gamma_1\le \frac 1{50}$
such that the following holds  for all  $m\geq m_0$. \\ Suppose $G$ is a $\gamma_0$-special $(m+1)$-vertex graph of minimum degree at least~$\lfloor\frac 23m\rfloor$, and  suppose~$T$ is a tree with $m$ edges such that none of its vertices is adjacent to more than $\beta m$ leaves. Let $T^*$ be a $\gamma_1$-nice subtree of~$T$, with root~$t^*$, and let $L\subseteq V(T^*)\setminus\{t^*\}$. Assume there is a set $S\subseteq V(G)$ such that $|S|\le |L|-\lceil (\frac{\gamma_1}2)^4 m\rceil$. \\
  Assume that for any  $W\subseteq V(G)-S$ with $|W|\geq \gamma m$, there is an embedding $\phi_W$ of $T^*-Y$ into $G- S$, with $t^*$ embedded in $W$. Then there is a set $W\subseteq V(G)-S$ with $|W|\geq \gamma m$, and an embedding of $T-Y$  into $G-S$ that extends $\phi_W$. 
\end{lemma}

We prove Lemma~\ref{maya2} in Section~\ref{ma2}.

\subsection{Proving Lemma~\ref{theotherlemma}}\label{proofofthelemma}
  We now show how our four auxiliary lemmas imply Lemma~\ref{theotherlemma}.
 
\begin{proof}[Proof of Lemma~\ref{theotherlemma}]
First, we apply Lemma~\ref{maya2} to obtain four numbers  $\beta, \gamma_0, \gamma_1>0$ and $m_0^{Lem~\ref{maya2}}\in\mathbb N$. Next, we apply  Lemma~\ref{bruce} %twice, once for each input $\gamma_0, \gamma_1$ 
to obtain a number $m_0^{Lem~\ref{bruce}}$. Finally, we apply Lemma~\ref{maya1} with input $\gamma_0$ to obtain another integer~$m_0^{Lem~\ref{maya1}}$ as well as a number $\lambda >0$.

For the output of Lemma~\ref{theotherlemma}, we will take $$m_1:=\max\{m_0^{Lem~\ref{maya2}}, m_0^{Lem~\ref{bruce}}, m_0^{Lem~\ref{maya1}}, \frac{200}{\gamma_0}\},$$ and  $$\delta:=\min\{\beta, \lambda, 10^{-23}\}.$$ 

Now, consider an $m$-good graph~$G$, and a tree $T$ with $m$ edges as in the statement of Lemma~\ref{theotherlemma}. 
Use Lemma~\ref{cutthetree} together with Lemma~\ref{bruce}, once for each input $\gamma_0$, $\gamma_1$,  to obtain, for $i=0,1$, a $\gamma_i$-nice tree $T^*_i$ with root $t^*_i$, and sets $S_i$, $L_i$ satifying $$|S_i|\  \le \ |L_i|-(\frac{\gamma_i}2)^4m.$$ 

Moreover, for $i=0,1$, there are embeddings of $T^*_i-L_i$ into $G-S_i$ that map the vertex $t^*_i$ to any given vertex, except possibly the universal vertex of $G$. 
Furthermore, Lemma~\ref{bruce} guarantees that, in order to embed $T$ into $G$, we only need to extend, for either $i=0$ or $i=1$, the  embedding of $T^*_i-L_i$ given by the lemma  to an embedding of all of $T-L_i$ into $G-S$.

For this, we will use Lemmas~\ref{maya1} and~\ref{maya2}.
More precisely, if $G$ is not $\gamma_0$-special, then we can apply Lemma~\ref{maya1} to $G$ with sets $S_0$ and $L_0$, together with the tree $T^*_0$. 
If $G$ is $\gamma_0$-special,  we can apply Lemma~\ref{maya2} to $G$ with sets $S_1$ and~$L_1$, together with the tree $T^*_1$. 
 This finishes the proof of the lemma. 
\end{proof}

\subsection{Proof of Lemma~\ref{cutthetree}}\label{gammalemma}

We finish Section~\ref{proofofthel} by  
  giving the short proof of Lemma~\ref{cutthetree}. 

\begin{proof}[Proof of Lemma~\ref{cutthetree}]
As an auxiliary measure, we momentarily fix any leaf~$v_L$ of the given tree~$T$ as the root of $T$. 
Next, we choose a vertex $t^*$ in  $T$ having at least  $\lceil\frac{\gamma m}{2}\rceil$ descendants, such that it is furthest from  $v_L$ having this property.

Then, each component  of $T-t^*$ that does not contain $v_L$ has size at most~$\lceil\frac{\gamma m}2\rceil$.
 So, there is a subset $S^*$ of these  components  such that 
$$\Big\lceil\frac{\gamma m}2\Big\rceil\ \leq \ \sum_{S\in S^*}|S|\ \leq \ \gamma m.$$

Now, consider the tree $T^*$ formed by the union of the trees in $S^*$ and the vertex~$t^*$. Clearly, $T^*$  fulfills items (i) and~(ii) of  Definition~\ref{gamma-nice}.
If~$T^*$ contains at least $\lceil\frac{\gamma m}{40}\rceil$ leaves of $T$, then $T^*$ is $\gamma$-nice  of type 2, and we are done. 

Otherwise,  $T^*$ has at most $\lfloor\frac{\gamma m}{40}\rfloor$ leaves, and a standard calculation shows that $T^*$ has at most $\lfloor\frac{\gamma m}{40}\rfloor$ vertices of degree at least $3$. Delete these vertices from $T^*$. It is easy to see that this 
leaves us with a set of at most $\frac{\gamma m}{20}$ paths, together containing at least $\frac{19}{40}\gamma m$ vertices. All  vertices of these paths have degree at most $2$ in~$T$. 
Deleting at most four vertices on each path we can ensure all paths have lengths  divisible by five, and together contain at least $\frac{19}{40}\gamma m-4\cdot \frac{\gamma m}{20} \ \ge \ \frac{\gamma m}{4} +5$ vertices. Dividing each of the paths into paths of length five we obtain  a set ${\cal P}$  of at least~$\lceil\frac{\gamma m}{20}\rceil$ disjoint 
paths in~$T^*$. So, $T^*$ is $\gamma$-nice of type~1. 
\end{proof}

\section{The proof of Lemma \ref{easybruce}}\label{br1}

This section is devoted to the proof of the following lemma, which proves Lemma~\ref{bruce}  for all $\gamma$-nice trees of type~1.

\begin{lemma} 
\label{easybruce}
 There is an $m_0\in\mathbb N$ such that the following holds for all $m \ge m_0$, and for all $\gamma>0$ with $\frac 2{10^{7}}\le\gamma<\frac 1{30}$. \\
Let $G$ be an $m$-good graph. Let $T$ be a tree with $m$ edges, such that no vertex of $T$ is  adjacent to more than $\frac m{10^{23}}$ leaves. Let $T$ have a  $\gamma$-nice
 subtree $T^*$ of type 1, with root $t^*$. \\
 Then there are sets $L\subseteq V(T^*)\setminus\{t^*\}$ and $S\subseteq V(G)$  satisfying
$|S| \leq |L|-(\frac{\gamma}2)^4 m$. Furthermore, for any $w\in V(G)-S$, there is  an embedding of $T^*-L$ into  $G-S$, with $t^*$ embedded in $w$, such that any embedding of $T-L$ into $G-S$ extending our embedding of $T^*-L$  can be extended to an embedding of all of $T$ into $G$.
\end{lemma}

In the proof of Lemma~\ref{easybruce}, some random choices are going to be made, and in order to see we are not far from the expected outcome, it will be useful to have the well-known Chernoff bounds at hand (see for instance~\cite{McD89}). For the reader's convenience let us state these bounds here.

Let $X_1, \ldots , X_n$ be independent random variables satisfying $0 \leq X_i \leq 1$. Let $X = X_1 + \ldots + X_n$ and set $\mu:=\mathbb E[X]$. Then for any $\eps\in (0,1)$, it holds that
\begin{equation}\label{chernoff} \mathbb P [X \geq(1+\eps)\mu] \leq e^{-\frac{\eps^2}{2+\eps} \mu}\text{ \ \ \ and \ \ \
}\mathbb P [X \leq(1-\eps)\mu] \leq e^{-\frac{\eps^2}{2} \mu}.
\end{equation}

%We may need another lemma here:
%\begin{lemma}
%
%\end{lemma}

%\smallskip

We are now ready for the proof of Lemma~\ref{easybruce}.

\begin{proof}[Proof of Lemma~\ref{easybruce}]
We choose $m_0=10^{25}$. Now assume that for some $m\geq m_0$, we are given an $m$-good graph $G$,  and a tree $T$ with $m$ edges such that none of its vertices is adjacent to  more than $10^{-23} m$ leaves. We are also given a $\gamma$-nice subtree $T^*$ of $T$, with root $t^*$, and a set~$\mathcal P$ of disjoint paths of length five such that $$|\mathcal P|=\lceil\frac{\gamma m}{20}\rceil,$$ for some $\gamma$ as in the lemma. %We will work with a subset $\mathcal P$ of $\mathcal P^*$  of size $$|\mathcal P|=\lceil\frac m{10^8}\rceil.$$

We now define
 $L$ as the set that consists of the fourth vertex (counting from the vertex closest to $t^*$) of each of the paths from $\mathcal P$.  Clearly,
\begin{equation}\label{sizeofL}
|L|=\lceil\frac{\gamma m}{20}\rceil \ge \lceil\frac m{10^8}\rceil,
\end{equation}
by our assumptions on $\gamma$.

In order to prove Lemma~\ref{easybruce}, we need to do three things. First of all, we need to find a set $S\subseteq V(G)$ of size at most $|L|-(\frac{\gamma}2)^4 m$. Then, given any  vertex  $w\in V(G)-S$, we have to embed $T^*-L$ into $G-S$, with $t^*$ going to~$w$. Finally, we need to make sure that any extension of this embedding to an embedding of all of $T-L$ into $G-S$ can be completed to an embedding of all of $T$.

It is clear that for the last point to go through, it will be crucial to have chosen both $S$ and the set $N$ of the images of the neighbours of the vertices in~$L$ carefully, in order to have the necessary connections between $N$ and $S$. Our solution is to choose both $S$ and $N$ randomly. More precisely, choose a set $S$ of size 
\begin{equation}\label{sizeofS}
|S|= |L|-\lceil(\frac{\gamma}2)^4 m\rceil
\end{equation}
 uniformly and independently at random in $V(G-w)$.  Also,  choose a set $N$ of size
\begin{equation}\label{sizeofN} 
|N|=2|L|
\end{equation}
 uniformly and independently at random in $V(G-w-S)$. 

\medskip

Now, we can proceed to embed $T':=T^*-L$ into $G-S$. We will start by embedding the neighbours of vertices in $L$ arbitrarily into $N$. Let us keep track of these by calling $n_1(x)$ and $n_2(x)$ the images of the neighbours of $x$, for each $x\in L$. 

Next, we embed $t^*$  into~$w$, and then proceed greedily,  using a breadth-first order on $T^*$ (skipping the vertices of $L$ and those already embedded into $N$). Each vertex we embed has at most two neighbours that have been embedded earlier (usually this is just the parent, but parents of vertices embedded into~$N$ have two such neighbours, and the root of $T'$ has none). So, since~$G$ has minimum degree  at least $\lfloor\frac{2m}3\rfloor$ and given the small size of $T'$, we can easily embed all of $T'$ as planned.

\medskip

It remains to prove that any extensions of this embedding can be completed to an embedding of all of $T$. This will be achieved by the following claim, which finishes the proof of Lemma~\ref{easybruce}.

\begin{claim}
\label{fewleavescanfinishoff}
For any set $R\subseteq V(G)$ of $|L|-|S|$ vertices, there is a bijection between~$L$ 
and $S \cup R$ mapping each vertex  $x\in L$ to a common neighbour of $n_1(x)$ and~$n_2(x)$. 
\end{claim}
%
%and:
%
%\begin{lemma}
%\label{canpartition}
%Letting $U_1,...U_p$ be the non-singleton components of $T-E(T^*)$, we can partition 
%$G-V(H) -S$ into $p$ parts $S_1,...,S_p$ such that 
%each $S_i$ has between  $|U_i|(1+\frac{\gamma}{2})$ and $|U_i|(1+\frac{3\gamma}{2})$ vertices 
%and every vertex of $G$ has  at least $\frac{13|U_i|}{20}$ neighbours in $U_i$. 
%\end{lemma}
%
%Proof of Lemma \ref{canpartition}: Just randomly assign each vertex of $G-V(H)-S$ to 
%each $S_i$ with probability proportional to the size of $|U_i|$ and apply Chernoff. 

In order to prove Claim~\ref{fewleavescanfinishoff}, we 
define an auxiliary bipartite graph $H$ having $V(G-w)$ on one side, and $L$ on the other. We put an edge between $v\in V(G-w)$ and $x\in L$ if $v$ is adjacent to both $n_1(x)$ and $n_2(x)$. We are interested in the subgraph $H'$ of $H$ that is obtained by restricting the $V(G-w)$-side to the set $S\cup R$ (but sometimes it is enough to consider degrees in $H$).

 By the minimum degree condition on $G$, the expectation of the degree in~$H$ of any vertex  $v\in V(G-w)$ is $$\mathbb E(deg_H(v))\geq(\frac{199}{300})^2|L|,$$
 since $v$ has at least $\lfloor\frac{2}{3}m-1\rfloor\geq \frac{199}{300}m$ neighbours in $G-w$, and thus, for any given $x\in L$, each $n_i(x)$ is adjacent to $v$ with probability at least $\frac{199}{300}m$.
 Therefore, the
 probability that all vertices of $G$ have degree at least $$d:=(\frac{198}{300})^2|L|$$ is bounded from below by
 \begin{align}
 \notag \mathbb P[\delta(G)\geq d]\ & \geq \ 1-\sum_{v\in V(G-w)} \mathbb P[deg_H(v)< d] \\  \notag & \geq \ 1- (m+1)\cdot e^{-(\frac{397}{199\cdot 300})^2\frac {|L|}{2}}\\ \notag & \geq \ 0.9999,
 \end{align}
 where  we used~\eqref{chernoff} (Chernoff's bound) with $\varepsilon=\frac{199^2-198^2}{199^2}=\frac{397}{199^2}$, our bound on the size of $L$ as given in~\eqref{sizeofL} and the fact that $m\geq 10^{25}$.
 
Furthermore, since $G$ has minimum degree at least~$\lfloor \frac 23m\rfloor$, we know that for each $x\in L$, vertices $n_1(x)$ and $n_2(x)$ have at least $\frac 13m-3$ common neighbours in $G-w$. Therefore, every vertex of $L$ has degree at least $\frac{1}{3}m-3$ in~$H$. However, we are interested in the degree of these vertices  into the set~$S$. For a bound on this degree, first note that the expected degree of any vertex of $L$ into the set~$S$ is bounded from below by $\frac{999}{3000}|S|$. Now again apply~\eqref{chernoff}  (Chernoff's bound), together with the fact that $|S|\geq 10^{17}$, to obtain that with probability greater than $0.9999$, every element  of $L$ is incident to 
at least $\frac{998}{3000}|S|$ vertices of $S$. 

Resumingly, we can say that with probability greater than $0.999$ we chose the sets $S$ and $N$ such that the resulting graph $H$ obeys the following degree conditions:
\begin{enumerate}[(A)]
\item\label{likelyA} the minimum degree of $V(G-w)$ into $L$ is at least $(\frac{198}{300})^2|L|$; and
\item\label{likelyB} the minimum degree of $L$ into $S$ is at least $\frac{998}{3000}|S|$.
\end{enumerate}

Let us from now on assume that we are in the likely situation that both~\eqref{likelyA} and~\eqref{likelyB} hold. 

Further, assume 
 there is no matching from $S \cup R$ to $L$ in $H'$.  Then by Hall's theorem\footnote{Hall's theorem can be found in any standard textbook, it states that a bipartite graph with bipartition classes $A$ and $B$ either has a matching covering all of $A$, or there is an `obstruction': a set $A'\subseteq A$ such that $|N(A')|<|A'|$.}, there is a partition of $L$ into sets $L'$ and $L''$ and a partition of $S \cup R$ into sets $J'$ and $J''$ such that
there are no edges from $L'$ to $J''$, and such that $$|J'|<|L'|\text{ \  and \ }|L''|<|J''|.$$ 

Since $J''\neq \emptyset$, and since by~\eqref{likelyA}, each vertex in $J''$ has degree at least $(\frac{198}{300})^2|L|$ into~$L$, and thus into $L''$, we deduce that
\begin{equation}\label{J''}
|J''|> |L''|\geq(\frac{198}{300})^2|L|.
\end{equation}

Since also $L'\neq \emptyset$, and by~\eqref{likelyB}, each of its elements has at least $\frac{998}{3000}|S|$ neighbours in $S\cap J'$, we see that 
$$|L'|>|J'|\geq \frac{998}{3000}|S|.$$ Thus, using~\eqref{sizeofL} and~\eqref{sizeofS}, as well as our upper bound on $\gamma$, we can calculate that
\begin{equation}\label{sizeofL''}
|L''| \ =\ |L|-|L'|\ \leq \ |S|+\lceil (\frac{\gamma}{2})^4 m \rceil -\frac{998}{3000}|S| \ \leq \ \frac{2003}{3000}|S|.
\end{equation}

Let us iteratively define a subset $S^*$ of $S \cap J''$  as follows. We start by putting an arbitrary  vertex $v_0\in S \cap J''$ into $S^*$, and while there is a vertex of $S \cap J''$ whose neighbourhood contains 
$\frac{m}{1000 \log m}$ vertices which are not in the neighbourhood of $S^*$, we augment  $S^*$ by adding any  such vertex $v$ that
maximises $N(v)-N(S^*)$. We stop when there is no suitable vertex that can be added to $S^*$. 
Note that  $|S^*|\leq 1000 \log m$. 

Our plan is to show next that the set $S^*$ has certain properties which are unlikely to be had by {\it any} set having certain other properties that $S^*$ has (for instance, having size at most $1000 \log m$). More precisely, the probability that a set like $S^*$ exists will be bounded from above by $0.005$. This will finish the proof of Claim~\ref{fewleavescanfinishoff}, as we then know that with probability at least $0.99$ we chose sets~$S$ and $N$ such that in the resulting graph $H'$, the desired matching  exists, and thus Claim~\ref{fewleavescanfinishoff} holds.

So, let us define $\mathcal Q$ as the set of all subsets of $V(G-w)$ having size at most $1000 \log m$. For each $Q\in \mathcal Q$, let $V_1(Q)$ be the set consisting of all 
vertices of~$G-w$ which have less than $\frac{m}{1000 \log m}$ neighbours outside $N(Q)$ (in the graph~$G-w$). 

Finally, let $\mathcal Q'\subseteq \mathcal Q$ contain all $Q\in\mathcal Q$ for which
\begin{equation}\label{V1Q}
 \frac{m}{10^9}\ \leq \ |V_1(Q)|\ \leq \ \frac m3 + \frac m{\log m}+2.
\end{equation}
Observe that, for $Q\in\mathcal Q'$ fixed, the expected size of $V_1(Q)\cap S$ is $$\mathbb E[V_1(Q)\cap S]=|V_1(Q)|\cdot \frac{|S|}{m}$$ because $S$ was chosen at random in $G-v$. So by~\eqref{sizeofS} and~\eqref{sizeofL}, and by~\eqref{V1Q},  we see that 
 \begin{equation}\label{V1QcapS}
 \frac 12\cdot \frac{m}{10^{17}} \ \leq \ \mathbb E[V_1(Q)\cap S]\ \leq \ \frac{|S|}3 + \frac{|S|}{\log m} +2 \ \leq \ \frac {38}{100}|S|,
\end{equation}
where the last inequality follows from the fact that $m \geq 10^{25}$.
 Now, we can use~\eqref{chernoff} (Chernoff's bound) and the first inequality of~\eqref{V1QcapS} to bound the probability that
$|V_1(Q)\cap S|$ exceeds its expectation by a factor of at least $\frac{20}{19}$ as follows: $$\mathbb P\Big [|V_1(Q)\cap S|\geq \frac{20}{19} \cdot\mathbb E[V_1(Q)\cap S] \Big] \ \leq \ e^{-\frac{\mathbb E[V_1(Q)\cap S]}{820}}\ \leq \ e^{-\frac{m}{164\cdot 10^{18}}}\ \leq \ \frac{0.001}{m^{\log m}}.$$
Since by~\eqref{V1QcapS}, we know that $$\frac{20}{19}\cdot \mathbb E[V_1(Q)\cap S]\  < \ \frac{41}{100}|S|,$$ and since $|Q|\leq m^{\log m}$ for each $Q\in\mathcal Q$, we can 
deduce that

\begin{equation}\label{unlikely}
\mathbb P\Big [ \ \exists Q\in\mathcal Q' \text{ with } |V_1(Q)\cap S|\geq \frac{41}{100} |S| \ \Big ] \ \leq \ 0.001.
\end{equation}

\smallskip

Now, let us turn back to the set $S^*$. First of all, we note that by the definition of~$S^*$,  we have $S \cap J''\subseteq V_1(S^*)$. Thus, we can use~\eqref{J''} and~\eqref{sizeofS} to deduce that
\begin{align}\label{S^*1}
|V_1(S^*)\cap S|\ & \geq \ |J''|-|R|\notag \\ & \ \geq \ (\frac{198}{300})^2|L| -\lceil(\frac{\gamma}2)^4 m\rceil\notag \\ &\ \geq \ (\frac{197}{300})^2|S|\notag \\ &\ > \ \frac{43}{100}|S|.
\end{align} 
So, by~\eqref{sizeofL} and~\eqref{sizeofS},  the first inequality of~\eqref{V1Q} holds for $Q=S^*$.

\smallskip

For a moment, assume that $N(S^*)\leq \frac{999}{1000}m$. Then, also the second inequality of~\eqref{V1Q} holds for $Q=S^*$, as otherwise, each of the at least $\frac m{1000}$ vertices of $V(G-w)\setminus N(S^*)$ sees at least $\frac m{\log m}$ vertices of $V_1(S^*)$, and so, by the definition of~$S^*$, we have that
\begin{align*} \frac{m}{1000} \cdot \frac m{\log m}\ &\ \leq e(V_1(S^*),V(G-w)\setminus N(S^*))\\ &< \frac m{1000\log m} \cdot |V_1(S^*)|\\ & \leq \frac{m^2}{1000\log m},
\end{align*} 
a contradiction.  Hence $S^*\in\mathcal Q'$. But then, according to~\eqref{unlikely}, we know that~\eqref{S^*1}
is not likely to happen. So, with probability at least $0.998$, we  chose $S$ in a way that all three of~\eqref{likelyA}, \eqref{likelyB},  and
\begin{enumerate}[(A)]
\addtocounter{enumi}{2}
\item \label{N(S^*1)big}\label{likelyC} $|N(S^*)|\geq \frac{999}{1000}m$
  \end{enumerate} 
 hold. We will from now on assume that we are in this likely case.
 
 \smallskip
  
Consider the set $\mathcal Q''$ which consists of all sets $Q\in\mathcal Q$ for which the first inequality in~\eqref{V1Q} holds, and for which $|N(Q)|\geq \frac{999}{1000}m$. By~\eqref{S^*1} and by~\eqref{N(S^*1)big}, $S^*\in \mathcal Q''$.
  
Call $\mathcal Q''_+$ the set of all $Q\in\mathcal Q''$ for which at least one of the following holds:
\begin{itemize}
\item $Q$ has a vertex of degree at least $\frac{2m}3+\frac m{100}$; or
\item $Q$ has two vertices $v, v'$ such that each sees at least $\frac m{100}$ vertices outside the neighbourhood of the other one.
\end{itemize}
 We are going to show that the sets $Q\in \mathcal Q''_+$ typically have larger neighbourhoods in $L$ than $S^*$ has, and will thus be able to conclude that $S^*\notin\mathcal Q''_+$, which will be crucial for the very last part of the proof.
  
  For this, let $X(Q)$ be the set of unordered pairs $\{v,v'\}$ of distinct vertices which have a common neighbour in $Q$, for each $Q\in\mathcal Q''$. Then, because of the minimum degree condition we imposed on the graph~$G$, we know that each vertex $v\in N(Q)$ is in at least $\lfloor \frac{2m}{3}\rfloor -2$ pairs of $X(Q)$. So, since $N$ was chosen at random in $V(G-w)$, and because of the definition of $\mathcal Q''$, we know that for any fixed set $Q\in\mathcal Q''$, and any fixed vertex $x\in L$, the probability that $n_1(x)$ and $n_2(x)$ have a common neighbour in $Q$ can be bounded as follows:
\begin{align*}
\mathbb P[\{n_1(x),n_2(x)\}\in X(Q)] 
 \geq \frac{\frac{999m}{1000}\cdot (\lfloor \frac{2m}{3}\rfloor -2 )}{m^2}.
\end{align*}

However, if we take  any fixed $Q\in\mathcal Q''_+$, and any fixed $x\in L$, the bound becomes
\begin{align*}
\mathbb P[\{n_1(x),n_2(x)\}\in X(Q)] 
 & \geq \frac{\frac{999m}{1000}(\lfloor \frac{2m}{3}\rfloor -2)+\min\{(\frac{2m}3+\frac m{100})\frac m{100},(\frac m3-2) \frac m{100}\} }{m^2}\\ &\geq \frac{669}{1000},
\end{align*}

where the two entries in the minimum stand for the two scenarios that may cause the set $Q$ to belong to  $\mathcal Q''_+$. In order to to see the term for  the second scenario, observe that vertices  $v$ and $v'$ have at least $\frac m3 -2$ common neighbours, and each of  these neighbours belongs to at least $\lfloor\frac {2m}3\rfloor-2+\frac m{100}$ pairs of $X(Q)$.

Therefore, fixing $Q\in\mathcal Q''_+$, and letting $L(Q)$ denote the sets of all $x\in L$ with $\{n_1(x),n_2(x)\}\in X(Q)$, we know that the expected size of $L(Q)$ is bounded by
\begin{align*}
\mathbb E\Big[|L(Q)|\Big]  \geq \frac{669}{1000}|L|.
\end{align*}

As above, we can apply the Chernoff bound~\eqref{chernoff} to see that with very high probability, $|L(Q)|$ is not much smaller than its expectation: 
 $$\mathbb P\Big [|L(Q)|\leq \frac{668}{669} \cdot\mathbb E[|L(Q)|] \Big] \ \leq \ e^{-\frac{\mathbb E[|L(Q)|]}{2\cdot 669^2}}\ \leq \ e^{-\frac{|L|}{2\cdot 10^6}}\ \leq \ e^{-\frac{m}{2\cdot 10^{14}}}\ \leq \ \frac{0.001}{m\log m},$$
where we use~\eqref{sizeofL} and the fact that $m\geq 10^{25}$. So with probability at least $0.997$, we have chosen $N$ in a way that~\eqref{likelyA},~\eqref{likelyB},~\eqref{likelyC}, and also

\begin{enumerate}[(A)]
\addtocounter{enumi}{3}
\item \label{likelyD} $|L(Q)| \ > \ \frac{668}{1000}|L|\ = \ \frac{2004}{3000}|L|$ \ for every $Q\in \mathcal Q''_+$
\end{enumerate}
hold.

Because of~\eqref{sizeofL''} (and~\eqref{sizeofS}), and since $L''\supseteq L(S^*)$, this means that $$S^*\notin \mathcal Q''_+.$$
In particular, the degree of $v_0$  (in $G-w$) is less than $\frac{2m}{3}+\frac{m}{100}$, and each vertex of $S^*$ has less than $\frac m{100}$ neighbours outside $N(v_0)$. Moreover, by the choice of~$S^*$, we can deduce that 
\begin{equation}\label{final}
\text{every vertex in $S\cap J''$ has  less than $\frac m{100}$ neighbours outside $N(v_0)$. }
 \end{equation}
 
 By~\eqref{sizeofS} and by~\eqref{sizeofL''}, and since $|R|=|L|-|S|$, we know that
\begin{align}\label{final2}
|S\cap J''|\  \geq \ |J''|-|R| \ \geq \ (\frac{198}{300})^2|L| -\lceil(\frac{\gamma}2)^4 m\rceil\ > \ \frac{2}{5}|S|.
\end{align} 
Fix a subset $Z$ of size $\frac m4$ of $G-w-N(v_0)$, and
let us look at the average degree $d$ of the vertices of $Z$ into $S\cap J''$. We have
\begin{align*}
d\cdot \frac m4  = \sum_{v\in Z}deg(v,S\cap J'') = \sum_{v\in S\cap J''}deg(v,Z)\leq \frac{m\cdot |S\cap J''|}{100},
\end{align*} 
where for the last inequality we used~\eqref{final}. Thus $$d\leq \frac{|S\cap J''|}{25}.$$
Now use~\eqref{final2} to see that the average degree of the vertices of $Z$ into $S$ is bounded from above by $|S|-\frac{48}{125}|S|<(\frac 23-\frac 3{100})|S|.$ This means that there must be at least one vertex in~$Z$, say the vertex $z$, which has degree at most $(\frac 23-\frac 3{100})|S|$ into~$S$. However, by Chernoff's bound~\eqref{chernoff}, and since the expected degree of any vertex of $G-W$ into $S$ is at least $(\frac 23-\frac 1{1000})|S|$, we know that this would only happen with probability at most $0.001$. So we can assume we are in a situation where no such vertex $z$ exists, and reach a contradiction, as desired.

\smallskip

Resumingly, we know that with probability at least $0.995$, our choice of~$S$ and $N$ guarantee that a set $S^*$ as above does not exist in the resulting auxiliary graph $H'$, and thus, Hall's condition holds in $H'$. This means we find the desired matching, which finishes the proof of Claim~\ref{fewleavescanfinishoff}, and with it the proof of Lemma~\ref{easybruce}.

\end{proof}

\section{The proof of Lemma \ref{hardbruce}}\label{br2}

This section is devoted to the proof of the following lemma, which proves Lemma~\ref{bruce}  for all $\gamma$-nice trees  of type~2. So, since  $\gamma$-nice trees of type 1 are covered by Lemma \ref{easybruce}, this finishes the proof of Lemma~\ref{bruce}.

\begin{lemma}
\label{hardbruce}
There is an $m_0\in\mathbb N$ such that the following holds for all $m \ge m_0$, and all $\gamma>0$ with  $\frac 2{10^{7}}\le \gamma<\frac 1{30}$. \\
Let $G$ be an $m$-good graph, with universal vertex $w$. Let $T$ be a tree with $m$ edges, such that no vertex of $T$ is  adjacent to more than $\frac m{10^{23}}$ leaves. Let $T$ have a  $\gamma$-nice
 subtree $T^*$ of type 2, with root $t^*$. \\
 Then there are  sets $L\subseteq V(T^*)\setminus\{t^*\}$ and $S\subseteq V(G)$  satisfying
$|S| \leq |L|-(\frac{\gamma}2)^4 m$. Furthermore, for any $w'\in V(G)-(S\cup\{w\})$, there is  an embedding of $T^*-L$ into  $G-S$, with $t^*$ embedded in $w'$, such that any embedding of $T-L$ into $G-S$ extending our embedding of $T^*-L$  can be extended to an embedding of all of $T$ into $G$.
\end{lemma}

 In the proof of Lemma \ref{hardbruce} we will use Azuma's inequality  which can be found for instance in~\cite{McD89}). This well-known inequality states that for any sub-martingale $\{X_0, X_1, X_2,\ldots\}$ which for each $k$  almost surely satisfies $|X_k-X_{k-1}|<c_k$ for some $c_k$, we have that
 \begin{equation}\label{azuma}
 \mathbb P[X_n-X_0\leq -t]\ \leq \ e^{-\frac{t^2}{2\cdot\sum_{k=1}^nc_k^2}}
 \end{equation}
 for all $n\in \mathbb N_+$ and all positive $t$. 

\bigskip

Let us now give the proof of Lemma \ref{hardbruce}.

\begin{proof}[Proof of Lemma~\ref{hardbruce}]
We choose~$m_0\in\mathbb N$  large enough so that certain inequalities below are satisfied.

 Let $G$ be an $m$-good graph, with universal vertex $w$. Let $T$ be a tree with~$m$ edges, such that no vertex of $T$ is  adjacent to more than $\frac m{10^{23}}$ leaves.
We are also given a $\gamma$-nice subtree $T^*$ of $T$, with root $t^*$, and since $T^*$ is of type 2, there is a set $L^*\subseteq V(T^*)\setminus\{t^*\}$ of $|L^*|=\lceil\frac{\gamma m}{40}\rceil$ leaves of $T$. Instead of $L^*$, we will work with the set $L$ which is obtained from $L^*$ by deleting all neighbours of $t^*$. Cleary, 
$$|L|=\lceil\frac{\gamma m}{41}\rceil\ge \lceil\frac m{10^9}\rceil$$ leaves of $T$.

In order to prove Lemma~\ref{hardbruce}, it suffices to  find a set $S\subseteq V(G)$ satisfying
$|S| \leq |L|-(\frac{\gamma}2)^4 m$, to embed~$T^*-L$ into $G-S$, and show that any extension of this embedding to an embedding of $T-L$ into $G-S$ can be completed to an embedding of all of $T$ into $G$.

For this, let us define $t$  as the vertex of $T^*$ that is adjacent to most leaves from $L$. Define $\alpha$ so that $t$ is 
incident to $\lceil\alpha m\rceil$ leaves and call $L_t$ the set of these leaves. By the assumptions of the lemma,
\begin{equation}\label{1013}
\alpha\leq 10^{-23}.
\end{equation}

We now randomly embed $T^*-L$ in a top down fashion, where we start by putting $t^*$ in to $w'$. At each moment, when we embed a vertex $v\neq t$, we choose a 
uniformly random neighbour of the image  of the (already embedded) parent $p(v)$ of $v$. When we reach $t$, we embed $t$ into $w$, the universal vertex of $G$.  (This gives us some 
leeway when we later have to embed~$L$.)  We do not have to worry about the connection of $w$ to the image of $p(t)$ because of the universality of $w$.

For every $x\in L$, let us call $n(x)$ the image of $p(x)$.

\smallskip

Next, we pick a set $S$ of size $$|S|=|L|-\lceil (\frac {\gamma}{2})^4m \rceil$$
uniformly and independently at random in what remains of $G$.  It only remains to prove the following analogue of Claim~\ref{fewleavescanfinishoff} to finish the proof of Lemma~\ref{hardbruce}.

\begin{claim}
\label{fewleavescanfinishoff2}
For any set $R\subseteq V(G)$ of $|L|-|S|$ vertices, there is a bijection between~$L$ 
and $S \cup R$ mapping each vertex  $x\in L$ to a  neighbour of $n(x)$. 
\end{claim}

In order to prove Claim~\ref{fewleavescanfinishoff2}, consider a set $R$ of size $|L|-|S|$ such that there is no matching from $L$ to $S \cup R$  in the auxiliary bipartite graph $H$ which is defined as follows. The bipartition classes of this graph $H$ are $L$ and $S\cup R$, and every vertex  $x\in L$ is joined to all unoccupied neighbours of the image $n(x)$ of the parent of $x$ in $S\cup R$. Our aim is to derive a contradiction from the assumption that such a set $R$ exists.

Our first observation is that by Chernoff's bound~\eqref{chernoff} and by our assumption on the minimum degree of $G$, we know that with probability at least $0.999$, every vertex of $L$ has degree at least $(\frac{2}{3}-\frac{2}{10^4})|L|$ in $H$. 

Furthermore, as there is no matching from $L$ to $S \cup R$  in $H$,
we can apply Hall's theorem. This gives a partition of $L$ into sets $L'$ and $L''$ and a partition of $S \cup R$ into sets $J'$ and~$J''$ such that 
there are no edges from $L'$ to $J''$, and such that furthermore,  $$|J'|<|L'|\text{ \  and \ }|L''|<|J''|.$$

As $L'\neq\emptyset$, we know that  $|J'|\geq (\frac{2}{3}-\frac{2}{10^4})|L|$ and therefore,  \begin{equation}\label{J''512}|J''|\leq (\frac{1}{3}+\frac{2}{10^4})|L|.\end{equation}

Since~$L''$ contains all the children of $t$ (this follows from the definition of~$H$ and from the fact that $|J'|<m$), and because of the definition of $\alpha$, we know that $L''$ has size  at least~$\lceil\alpha m\rceil$ and hence  
\begin{equation}\label{J''alpham}
|J''|>\lceil\alpha m\rceil.  
\end{equation}

We now consider the set $V^*$ of vertices of $G$ which are adjacent to at most $(\frac{1}{3}+\frac{2}{10^4})|L|$
vertices of $L$ in $H$. (The vertices in $V^*$ are those that serve only for relatively few leaves in $L$ as a possible image.) Note that the size of $V^*$ depends on how we embedded $T^*-L$ (which was done randomly).
We plan  to show that 
\begin{equation}\label{V^*alpha}
\text{with probability $\ge 0.99$, we embedded $T^*-L$ such that $|V^*|<\alpha m$.}
\end{equation}
 Then, by~\eqref{J''alpham}  there is a vertex 
$v\in J''\setminus V^*$. As the neighbours of $v$ in $H$ are contained in $L''$, we get that $$|J''| >|L''| \ge (\frac{1}{3}+\frac{2}{10^4})|L|,$$ which 
is a contradiction to~\eqref{J''512}. This would prove  Claim~\ref{fewleavescanfinishoff2}.

\medskip

So, it only remains to show~\eqref{V^*alpha}. For this, we start by bounding
 the probability that a specific vertex $v$ is in $V^*$. Consider any vertex $p$ that is the parent of some subset~$L_p$ of $L$, and recall that $p$ was embedded randomly in the neighbourhood $N_p$ of the image of the parent of $p$. By our minimum degree condition on $G$, we know that $v$ is incident to at least $\frac{499}{1000}|N_p|$ vertices of $N_p$. 
 
 Hence,  the probability that $v$ is adjacent to $p$ in $G$, and thus to all of~$L_p$ in~$H$, is bounded from below by $\frac{499}{1000}$. Since $T^*-L$ is very small, this bound actually holds independently of whether $v$ is adjacent to  $L_{p'}$ for some other parent~$p'$. Therefore, 
 \begin{equation}\label{4991000}
 \text{the expected degree of $v$ into $L_p$ is at least $\frac{499}{1000}|L_p|$,}
  \end{equation}
   for each $p$.
  
 Our plan is to use Azuma's inequality (i.e., inequality~\eqref{azuma} above). 
 For this, order the set $P$ of parents $p$ of subsets $L_p$ of $L$ as above, writing $$P=\{p_1,p_2,\ldots,p_n\}.$$ For $1\leq i\leq n$, write $d_i$ for the degree of $v$ into $L_{p_i}$. Now, define the random variable $$X_k:=\sum_{1\leq i\leq k}d_i+ \frac{499}{1000} \cdot \sum_{k< i\leq n}|L_{p_i}|.$$ By~\eqref{4991000}, this is a sub-martingale. Observe that $$X_0=\frac{499}{1000}\cdot |L|$$ and $$X_n=\deg (v, L).$$
 
We set $c_k:=|L_{p_k}|$ for all $k\leq n$. Then $\sum_{k=1}^nc_k=|L|$, and furthermore, by our choice of the vertex  $t$ in the beginning of the proof of Lemma \ref{hardbruce}, we know that $$c_k\leq\alpha m,\text{ for all }k\leq n.$$ This, together with Azuma's inequality~\eqref{azuma}, tells us that the probability that $v$ is in $V^*$ can be bounded as follows:
\begin{align*}
\mathbb P[v\in V^*] \ & \leq \ \mathbb P\big [\deg (v, L)\leq \frac{336}{1000}|L|\big ]\\
& =\ \mathbb P[ X_n-X_0\leq -\frac{163}{1000}|L|]\\
& \leq e^{-\frac{(\frac{163}{1000}|L|)^2}{2\alpha m\cdot \sum_{k=1}^nc_k}}\\
& \leq e^{-\frac{163^2}{2\alpha \cdot10^{15}}}\\
& \leq e^{-\frac{1}{10^{11}\cdot \alpha}}.
\end{align*}
So, the expected size of $V^*$ is at most $m\cdot e^{-\frac{1}{10^{11}\cdot \alpha}}$. Using Markov's inequality we see that  the probability that $V^*$ contains more than $\alpha m$ vertices is bounded from above by $$\frac{e^{-\frac{1}{10^{11}\cdot \alpha}}}{\alpha}\leq 0.01,$$
where we used the fact that $\alpha\leq 10^{-23}$ by~\eqref{1013}. This proves~\eqref{V^*alpha}, and thus finishes the proof of Claim~\ref{fewleavescanfinishoff2}, and of Lemma \ref{hardbruce}.

\end{proof}

\section{ The proof of Lemma \ref{maya2}}\label{ma2}

The whole section is devoted to the proof of Lemma \ref{maya2}. We employ an ad-hoc strategy, which we briefly outline now. 

First, we clean up the $\gamma_0$-special host graph $G$, ensuring a convenient minimum degree between and inside the three sets $X_i$  (the witnesses to the fact that $G$ is $\gamma_0$-special, see Definition~\ref{gammaspecial}). 
Then,  given  the tree~$T$ with its $\gamma_1$-nice subtree $T^*$, rooted at $t^*$, we preprocess the part $T-T^*$ we have to embed. We do this by strategically choosing a small set $Z\subseteq V(T-(T^*-t^*))$, and divide the set $A$ of all components of $T-(T^*-t^*)-Z$ into two sets $A_1$ and $A_2$, which have certain useful  properties (see Claim~\ref{partition_of_A}). 
We embed $T-L$, extending the given embedding of $T^*-L$. 

We now distinguish three cases. In the first two cases, many elements of $A$ are three-vertex paths, and we embed them into $X_2\cup X_3$ and embed the rest into $X_1\cup X_3$. In the third case, there are not so many elements of $A$ that are three-vertex paths, and we will use the partition $A_1\cup A_2$ of $A$.
Components from sets $A_1$  will be embedded into $X_1\cup X_3$, and components from $A_2$ will be embedded into $X_2\cup X_3$.

Let us now formally give the proof of Lemma \ref{maya2}.

\paragraph{Setting up the constants and summarising the situation.}
For the output of Lemma \ref{maya2}, we choose  $$\beta:=\frac 1{10^{40}}\text{ \  
  and \ } m_0:=\frac 1{\beta^{100}},$$ and set $$\gamma_0:=\frac{2}{10^{7}}\text{ \  
  and \ } \gamma_1:=\frac 1{50}.$$
Now, assume we are given  a 
$\gamma_0$-special $(m+1)$-vertex graph $G$ of minimum degree at least $\lfloor\frac{2m}3\rfloor$, for some $m\geq m_0$, together with a tree $T$ having $m$ edges, such that none of the vertices  of $T$ is adjacent to more than $\beta m$ leaves. Assume
$T$ has  a $\gamma_1$-nice subtree $T^*$ rooted at  $t^*$,  and there are sets $L\subseteq V(T^*)\setminus\{t^*\}$ and $S\subseteq V(G)$ such that $|S|\le|L|-\lceil (\frac{\gamma_1}2)^4 m\rceil$. 

Furthermore,  for any large enough set $W$, it is possible to embed $T^*-L$  into a subset $\varphi(T^*-L)$ of $V(G)- S$, with $t^*$ going to $W$. (We will specify below which set~$W$ we will use.)
Once $T^*-L$ is embedded, our task is to embed the rest of $T-L$ into $G-(\varphi(T^*-L)\cup S)$. Observe that because of the discrepancy of the sizes of the sets~$L$ and $S$, we can count on an approximation of at least $\lceil (\frac{\gamma_1}2)^4\rceil$, that is, we  know our embedding will leave at least $\lceil(\frac{\gamma_1}2)^4m\rceil$ vertices of $G-(\varphi(T^*-L)\cup S)$ unused.

   \paragraph{Preparing $G$ for the embedding.}
   Since $G$ is $\gamma_0$-special, there are sets $X_1, X_2, X_3$ partitioning $V(G)$ such that 
   \begin{equation}\label{Xi_size}
   \frac m3-3\gamma_0 m\le |X_i|\le\frac m3+3\gamma_0 m
      \end{equation}
    for each $i=1,2,3$, and such that
   \begin{equation}\label{fewX1X2}
    \text{there are  at most $\gamma_0^{10} |X_1|\cdot|X_2|$  edges between $X_1$ and $X_2$. }
   \end{equation}
   
   Using  the minimum degree condition on $G$, and using~\eqref{fewX1X2}, an easy calculation shows that we can eliminate at most $\gamma_0^5m$ vertices from each of the sets~$X_i$, for $i=1,2$, so that the vertices of the thus obtained subsets~$X'_i\subseteq X_i$ each have degree at least $\lfloor\frac{2m}{3}\rfloor-\gamma_0^5|X_{3-i}|$ into $X'_i\cup X_3$, for $i=1,2$. 
 Then, because of~\eqref{Xi_size}, we can deduce that there are  at least $(1-6\gamma_0)|X'_i||X_3|$ edges between the sets $X'_i$ and $X_3$, for $i=1,2$. So, we can eliminate at most $2\cdot \sqrt{6\gamma_0}m$ vertices from $X_3$, obtaining a set $X'_3$, so that each of the vertices in $X'_3$ has degree at least $(1-6\sqrt{\gamma_0}) |X'_{i}|$ into $X'_i$, for $i=1,2$.
   
   Resumingly, we eliminated a few  vertices from each of the sets $X_1, X_2, X_3$ to obtain three  sets $X'_1, X'_2, X'_3$ satisfying 
         \begin{equation}\label{20b}
         |X'_i|\ge |X_i|-5\sqrt{\gamma_0} m
      \end{equation}
    such that  for $i=1,2$,  and  any vertex $v$ in $X'_3$,
        \begin{equation}\label{minnndeg}
    \text{the degree of $v$ into  $X'_i$ is at least $|X'_i|-3\sqrt{\gamma_0} m$.}
   \end{equation}
   Furthermore, for $i=1,2$,  for any $v\in X'_i$ and any $X\in\{ X'_i, X'_3 \}$, 
      \begin{equation}\label{minndeg}
    \text{the degree of $v$ into  $X$ is at least $|X|-6\sqrt{\gamma_0} m$.}
   \end{equation}
Indeed, in order to see~\eqref{minndeg} for $X=X'_i$, we use~\eqref{Xi_size} to calculate that
\begin{align*}
\deg(v, X'_i)=\deg(v, X'_i\cup X_3)-\deg (v, X_3)
&\ge \lfloor\frac{2m}{3}\rfloor-\gamma_0^5|X_{3-i}| - |X_3|\\
&\ge \lfloor\frac{m}{3}\rfloor-(\gamma_0^5+3\gamma_0)m\\
&\ge |X'_i|-6\sqrt\gamma_0 m,
\end{align*}
and for~\eqref{minndeg} for $X=X'_3$, we calculate similarly, also using~\eqref{20b}, to see  that
\begin{align*}
\deg(v, X'_3)&\ge \deg(v, X'_i\cup X_3)-|X'_i|- (|X_3|-|X'_3|)\\
&\ge \lfloor\frac{2m}{3}\rfloor-\gamma_0^5|X_{3-i}| - |X_i|- 5\sqrt{\gamma_0} m\\
%&\ge \lfloor\frac{m}{3}\rfloor-(\gamma_0^5+3\gamma_0+5\sqrt{\gamma_0})m\\
&\ge |X'_i|-6\sqrt\gamma_0 m.
\end{align*}

   \paragraph{Finding $Z$ and grouping the components.}
 Let us next have a closer look at the to-be-embedded  $T-T^*$. This forest might have relatively large components, which, for reasons that will become clearer below, might add unnecessary difficulties to our embedding strategy. In order to avoid these difficulties, we will now find a set $Z\subseteq V(T-(T^*-t^*))$ of up to three vertices  so that all components in $T-(T^*-t^*)-Z$ have controlled sizes, and can be grouped into convenient sets. (Note that $t^*$ may or may not lie in $Z$.)
 
 More precisely, our aim is to prove the following statement.
 \begin{claim}~\label{partition_of_A}
There are an independent set $Z\subseteq V(T)\setminus V(T^*-t^*)$ with $|Z|\le3$ and 
 a partition of the set $A$ of components of $T-(T^*-t^*)-Z$ into sets $A_1, A_2$ such that  for $i=1,2$,
\begin{enumerate}[(i)]
\item\label{verygoodZ} all but at most one $\bar T\in A$ has exactly one vertex $r_{\bar T}$ neighbouring $Z$;
 \item \label{zwischen} $\frac{m}{3}+\gamma_1 m\le |\bigcup_{\bar T\in A_i}V(\bar T)|\le \frac {2m}3-\gamma_1 m$; and
\item \label{los_panchosX}  if  $ |\bigcup_{\bar T\in A_i}V(\bar T)|\ge \frac{|\bigcup_{\bar T\in A}V(\bar T)|}{2}+\frac 1{\gamma_0}$, then each   $\bar T\in A_i$ has at least $\frac 1{\gamma_0}$ vertices.
 \end{enumerate}
 \end{claim}
 
 For proving Claim~\ref{partition_of_A}, we plan to use the following  folklore
 argument, and for completeness, we include its short proof. 
  \begin{claim}~\label{folklore}
{Every tree $D$ has a vertex $t_D$ such that each component of $D-t_D$  has size  at most $ \frac{|D|}2$.  }
 \end{claim}
 
\begin{proof}
In order to see Claim~\ref{folklore}, temporarily root $D$ at any leaf vertex $v_L$. Let~$t_D$ be a vertex that is furthest from $v_L$ having the property that $t_D$ and its descendants constitute a set of at least 
$\frac{|D|}2$ vertices. Then each component of $D-t_D$, including the one containing $v_L$, has at most $\frac{|D|}2$ vertices. 
\end{proof}
%
%   SAME THING with DIFFERENT CUT RATIOS
%
% \begin{claim}~\label{folklore}
%{For every tree $D$, and  for any number $c\in [\frac 12, 1]$, there is a vertex $t_{c}(D)$ in $D$ such that each but at most one component of $D-t_{c}(D)$  has size  at most $ (1-c) |D|$. Furthermore, if there is a component of size greater than $ (1-c) |D|$, then this component has size at most $c|D|$. }
% \end{claim}
% 
%\begin{proof}
%In order to see Claim~\ref{folklore}, temporarily root $D$ at any leaf vertex $v_L$. Let~$t_c$ be a vertex that is furthest from $v_L$ having the property that $t_c$ and its descendants constitute a set of at least 
%$(1-c) |D|$ vertices. Then~$D-t_c$ has at most one component $D'$, namely the one containing $v_L$, that could have more than $(1-c) |D|$ vertices. Furthermore, $D'$ has at most $c|D|$ vertices.
%\end{proof}
 
 We can now prove  Claim~\ref{partition_of_A}.

\begin{proof}[Proof of Claim~\ref{partition_of_A}]
Set $T':=T-(T^*-t^*)$ and apply Claim~\ref{folklore} to $T'$. We obtain a vertex $z$. Let $A_z$ be the set of all components of $T'-z$.

First assume there is a set $A_1\subseteq A_z$ with  
\begin{equation}\label{siri}
\frac{|\bigcup_{\bar T\in A}V(\bar T)|}{2}\le |\bigcup_{\bar T\in A_1}V(\bar T)|\le \frac 23m-\gamma_1 m.
\end{equation}
We can assume that $A_1$ is smallest possible with~\eqref{siri}. 
 This choice guarantees that either $A_1$ has no  components with at most $\frac 1{\gamma_0}$ vertices, or $ |\bigcup_{\bar T\in A_1}V(\bar T)|< \frac{|\bigcup_{\bar T\in A}V(\bar T)|}{2}+\frac 1{\gamma_0}$.
So  $Z:=\{z\}$, $A_1$ and $A_2:=A\setminus A_1$ are as desired. 

Now assume there is no set $A_1$ as in~\eqref{siri}. Then 
\begin{equation}\label{mark}
\text{there is not set $A'\subseteq A_z$ with $\frac{m}{3}+\gamma_1 m\le |\bigcup_{\bar T\in A'}V(\bar T)|\le \frac 23m-\gamma_1 m$}
\end{equation}
 (since if there was such a set $A'$, then either $A'$ or $A\setminus A'$ would qualify as~$A_1$).
We claim that  $T'-z$ has three components  $C_1$, $C_2$, $C_3$ such that  \begin{equation}\label{Ciok}
\frac m3- 2\gamma_1 m\le |C_i|\le\frac m3+ \gamma_1 m
\end{equation}
 for $i=1,2,3$  (additionally, $T'-z$ might have a set of very small components). Indeed, take a subset of $A'\subseteq A_z$ such that $|\bigcup_{\bar T\in A'}V(\bar T)|$ is maximised among all $A'$ with $|\bigcup_{\bar T\in A'}V(\bar T)|\le \frac 23m-\gamma_1 m$. Because of~\eqref{mark}, we know that $|\bigcup_{\bar T\in A'}V(\bar T)|<\frac{m}{3}+\gamma_1 m$, and moreover, for any component $C$ from $A\setminus A'$ we have that $|V(C)|\cup \bigcup_{\bar T\in A'}V(\bar T)|> \frac 23m-\gamma_1 m$. So, $|V(C)|> \frac m3-2\gamma_1 m$ for any such $C$, and Claim~\ref{folklore} implies that $|V(C)|\le \frac m2$. Hence there are exactly two such components, $C_1$ and $C_2$,  both of which fulfill~\eqref{Ciok}, and $A=A'\cup\{C_1, C_2\}$.
 
 A similar argument (using the fact that we did not choose $C_1$ together with a subset of $A'$ instead of choosing $A'$) gives that $A'$ contains a component $C_3$ for which~\eqref{Ciok} holds, and that 
 \begin{equation}\label{choco}
|V(T-T^*-C_1-C_2-C_3)|\le 3\gamma_1m.
\end{equation}

Apply Claim~\ref{folklore} to each of the three components $C_1$, $C_2$, $C_3$,  obtaining three vertices, $z_1$, $z_2$, $z_3$, such that for $i=1,2,3$, $z_i\in C_i$ and the components of $C_i-z_i$ have size at most $\frac m6+\frac {\gamma_1}2 m$. 
First assume that one of the vertices~$z_i$, say $z_1$,  is not adjacent to $z$. Then we set $Z:=\{z_1, z\}$. For $A_1$, we choose $C_2$ and some of the components of $C_1-z_1$, in a way that~\eqref{zwischen} of Claim~\ref{partition_of_A} holds for $A_1$. Let $A_2$ be the set of the remaining components of $T'-Z$. As before, we can ensure~\eqref{los_panchosX} by shifting some of the small components from one of $A_1$, $A_2$ to the other, until they have almost the same number of vertices, or the larger one has no small components. Note that most one component of $A=A_1\cup A_2$ is adjacent to both $z_1$ and $z$, which ensures~\eqref{verygoodZ}.

Now assume $z_iz$ is an edge, for each $i=1,2,3$.
Then we set 
 $Z:=\{z_1, z_2, z_3\}$. 
Observe that  the set $A$ of the components of $T'-Z$ is comprised of all components of $C_i-z_i$, for $i=1,2,3$, plus  a component containing $z$ and all vertices outside $C_1\cup C_2\cup C_3$. Each tree in $A$ has exactly one vertex neighbouring $Z$, as desired for~\eqref{verygoodZ}. Moreover, as these trees each have size at most $\frac m6+\frac {\gamma_1}2 m$, it is easy to group them into two sets  $A_1$ and $A_2$ fulfilling~\eqref{zwischen}, and as before, we can shift some of the small trees to ensure~\eqref{los_panchosX}.
\end{proof}

%  We now embed $T-T^{*}$, distinguishing three cases. For this, we define $A_L$ as the set of all componets  of size at least $\frac 1{\gamma_0}$, and set $A_S:=A\setminus A_L$. Furthermore, we say that a component of $A$ is {\em bad} 
%if it is isomorphic to a 3-vertex path whose middle vertex has degree $2$ in $T$. Let
% $B$ be the set of all bad components in $A$. Clearly,  $B\subseteq A_S$.
   We now embed $T-T^{*}$, distinguishing three cases.  For convenience, let us define $A^*\subseteq A$ as the set of those components that contain $t^*$ or are adjacent to more than one vertex of $Z$. By Claim~\ref{partition_of_A}~\eqref{verygoodZ}, $|A^*|\le 2$.
   Also, call 
    $\bar T\in A$ {\em bad} 
if $\bar T$ is isomorphic to a 3-vertex path whose middle vertex has degree $2$ in $T$. Let
 $B$ be the set of all bad components in $A\setminus A^*$.
 
  \paragraph{Embedding $T-T^{*}$ if $B$ is large.}
  
We show that if  
\begin{equation}\label{manyBP3}
\Big|\bigcup_{\bar T\in B}V(\bar T)\Big|>\frac m2,
\end{equation}
then we can embed $T-T^{*}$. Indeed, choose~$W$ as the set $X'_1$. That is, we let  $T^*-L$ be embedded into $\varphi(T^*-L)\subseteq (X_1\cup X_2\cup X_3)\setminus S,$ with $t^*$  embedded into any vertex from $X'_1$. We also embed all vertices from~$Z\setminus\{t^*\}$  into vertices from $X'_1$, respecting possible adjacencies to $t^*$. After doing this, we define, for $i=1,2,3$, $$S_i:=X'_i\setminus \Big(\varphi(T^*-L)\cup\varphi(Z)\cup S\Big).$$

Note that, for $i=1,2,3$, we have that 
$$\frac m3+3\gamma_0 m \ge \ |S_i| \
\ge \frac m3-3\gamma_0 m-5\sqrt\gamma_0 m-\gamma_1 m-4
\ge \frac m3-\frac{11}{10}\gamma_1 m,$$
because of~\eqref{Xi_size} and~\eqref{20b}.

Consider the following way to embed trees from $B$ into $S_2\cup S_3$: 
We  put the first vertex into $S_3$, the second vertex into $S_2$, and the third vertex into $S_3$. Embed as many trees from $B$ as possible in this way. Because of~\eqref{manyBP3}, and because of~\eqref{minnndeg} and~\eqref{minndeg}, we will use all but at most $3\gamma_0 m+3\sqrt\gamma_0 m$ of $S_2$ (and about half of $S_3$).

For the embedding of the remaining trees from $A$ (including those trees from $B$ that have not been embedded yet), note that  for any tree $\bar T\in A\setminus A^*$, we can embed the larger\footnote{If both classes have the same size, we  choose one class arbitrarily.} of its bipartition classes, minus the root $r_{\bar T}$ of $\bar T$, into $S_3$, and the other bipartition class into $S_1$. 
For the trees $\bar T\in A^*$ we can proceed similarly, only taking special care when embedding the parent $p$ of a vertex that is already embedded (either $t^*$ or a vertex from $Z$). We will embed $p$ into either  $S_1$ or $S_3$ (as planned), respecting the adjacencies to its two already embedded neighbours  (both of which see almost all of  $S_1\cup S_3$, so this is not a problem). 
Note that  if vertex $t^*$ belongs to the class that was chosen to be embedded into $S_3$, we `spoil' our plan by one vertex since $t^*$ has been embedded in $S_1$.

We embed trees from $A$ as long as we can in the manner described above. The next tree $\bar T$ is embedded with its larger bipartition class into $S_1\cup S_3$, and the smaller class into $S_1$, using as much as possible of $S_3$. Because of ~\eqref{minnndeg} and~\eqref{minndeg}, we will use all but at most $6\sqrt\gamma_0 m+1$ vertices of $S_3$. The remaining trees from $A$ are embedded into $S_1$, which finishes the embedding.
 
%  So from now on, we may assume   that
%
%\begin{equation}\label{not_too_much_in_B}
%\Big|\bigcup_{\bar T\in B}V(\bar T)\Big|\le\frac m2+\frac 32 \gamma_0m.
%\end{equation}

\paragraph{Embedding $T-T^{*}$ if $B$ is medium sized.}
We now show how to embed $T-T^{*}$ if
\begin{equation}\label{mediumB}
\frac 49m<\Big|\bigcup_{\bar T\in B}V(\bar T)\Big|\le\frac m2.
\end{equation}
In this case, we choose $W$ as the set $X'_3$ if $t^*\in Z$, that is, we let $T^*-L$ be embedded into $\varphi(T^*-L)\subseteq (X_1\cup X_2\cup X_3)\setminus S,$ with vertex $t^*$  embedded into a vertex $\varphi(t^*)$ from~$X'_3$. 
    If $t^*\notin Z$,  we choose $W$ as the set $X'_1$.
    
    Now assume that $T^*-L$ has been embedded. We next embed all vertices from~$Z\setminus\{t^*\}$ into $X'_3$, respecting possible adjacencies to $t^*$.    
   We then set, for $i=1,2,3$, $$S_i:=X'_i\setminus \Big(\varphi(T^*-L)\cup S\cup\varphi(Z\cup\{t^*\})\Big),$$
   and because of~\eqref{Xi_size} and~\eqref{20b},  we have 
   \begin{align}\label{SiiiiBBB}
   \frac m3+3\gamma_0 m \ \ge \  |S_i| & \ge \ 
 \frac m3-(3\gamma_0 +5\sqrt\gamma_0) m-(\gamma_1m -\lceil (\frac{\gamma_1}2)^4 m\rceil)-4\notag \\ 
  &\ge \ \frac m3-\frac{11}{10}\gamma_1m. 
\end{align}

We will now embed  some trees  $\bar T\in B$ in the following way. Embed the first and the third vertex of $\bar T$ into $S_2$, while the second vertex may go to either $S_2$ or $S_3$. We embed as many trees from $B$ as possible in this way, and  fill as much as possible of $S_2$ with them. Then, because of~\eqref{minnndeg},~\eqref{minndeg},~\eqref{mediumB} and~\eqref{SiiiiBBB}, we will have used all but at most $6\gamma_0m$ vertices of $S_2$, and we will also have used at least $\frac m9-3\gamma_0$ vertices of $S_3$. If we did not embed all of $B$ we have used about half of the set $S_3$, and we embed the few remaining trees from $B$ into $S_1$. We finish the embedding by putting all the remaining components into $S_1\cup S_3$, as follows.

Consider any tree $\bar T\in A\setminus (A^*\cup B)$, and let $r_{\bar T}$ denote its root. As the parent of $r_{\bar T}$ was embedded into $S_3$, we have to  embed $r_{\bar T}$ into~$S_1$,  but then we could either embed $\bar T-r_{\bar T}$  so that the even levels go to $S_1$, and the odd levels go to~$S_3$, or we could embed $\bar T-r_{\bar T}$ the other way around  (if there is enough  space). This means that for each $\bar T\in A$, we can embed its larger bipartition class, except possibly for $r_{\bar T}$, into $S_3$, and the rest into~$S_1$. Even better, since any vertex in $S_1$ is adjacent to almost all of $S_1$, we note that any of the vertices that went to $S_3$ could alternatively have been placed in $S_1$. 
Hence, we can embed~$\bar T$ such that for any given  $t\leq \lceil\frac{|\bar T|-1}2\rceil,$ exactly $t$ vertices go to~$S_3$, and the rest go to~$S_1$.

So, as long as there is reasonable space left in both sets $S_1$ and $S_3$, we know that for each tree $\bar T\in A\setminus A^*$ with $|V(\bar T)|\ge 5$,
 one can embed   two fifth of its vertices (or less, if desired) into $S_3$ (as $\frac 25 |V(\bar T)|\le \lceil\frac{|\bar T|-1}2\rceil$ for these trees). For trees $\bar T\in A\setminus A^*$ with $|V(\bar T)|< 5$, it is easy to see that  $\bar T\notin B$ ensures that at least half of its vertices can be embedded into $S_3$ (or less, if desired).
 
For the trees in $A^*$ we can argue analogously, except that the vertex $t^*$ is already embedded into the set $X'_1$, and any neighbour of a vertex from $Z$ is forced to go to $S_1$. Therefore we might have two vertices less than expected going to $S_3$, but this does not matter for the overall strategy.
Thus, we can embed all trees from $A\setminus B$ into $S_2\cup S_3$, which finishes the embedding in this case.

\paragraph{Embedding $T-T^{*}$ if $B$ is small.}
We finally show how to embed  $T-T^*$~if
\begin{equation}\label{smallB}
\Big|\bigcup_{\bar T\in B}V(\bar T)\Big|\le\frac 49m.
\end{equation}
   As in the previous case, we set $W:=X'_3$ if $t^*\in Z$ and  set $W:=X'_1$ otherwise. Without loss of generality, let us assume that $t^*$ lies in a component from $A_1$ (otherwise we rename $A_1$ and $A_2$).    
   
    Now assume that $T^*-L$ has been embedded. We now embed~$Z\setminus\{t^*\}$ into $X'_3$, respecting possible adjacencies to $t^*$. We then embed  the at most $4\beta m$ leaves adjacent to $Z\setminus\{t^*\}$ anywhere in $G$, using~\eqref{minnndeg} and~\eqref{minndeg}. 
For $i=1,2,3$, let $S_i$ be the set of all unused vertices from $X'_i\setminus S$.
   By~\eqref{Xi_size} and~\eqref{20b}, and since $\beta\ll\gamma_0$, we  calculate similarly as for~\eqref{SiiiiBBB} that 
   \begin{align}\label{Siiii}
   \frac m3+3\gamma_0 m \ \ge \  |S_i| %& \ge \  \frac m3-(3\gamma_0 +5\sqrt\gamma_0) m-(\gamma_1m -\lceil (\frac{\gamma_1}2)^4 m\rceil)-4-4\beta m\notag \\   &
  \ge \ \frac m3-\frac{11}{10}\gamma_1m. 
\end{align}

  We will next embed the components from $A$. 
  As in the previous case, we see that 
for any tree $\bar T\in A\setminus (A^*\cup B)$, for any $i\in\{1,2\}$, and for any  $t\leq \lceil\frac{|\bar T|-1}2\rceil$, we can  embed~$\bar T$ into $S_i\cup S_3$ with exactly $t$ vertices going to~$S_3$. For the trees in $A^*$ the same is true if we replace $t$ with $t-1$.

So, as above, the trees in $B$ can be embedded with a third of their vertices (or less, if desired) going to $S_3$.
 The trees in $A\setminus B$ having size less than $\frac 1{\gamma_0}$ can be embedded with two fifth of their vertices, or less, if desired, going to $S_3$. For the at most two trees in $A^*\setminus B$ the same is true, but we might have (in total) two vertices less in $S_3$. For the trees in $A$ having size at least $\frac 1{\gamma_0}$, however, we can work under the stronger assumption that half of their vertices (or less, if desired) may be embedded into $S_3$. 
This is so because there are at most $\gamma_0 m$ such trees, and hence for embedding their roots we will use at most $\gamma_0 m$ vertices, which is small enough to play no role in the calculations.

We will now see that the above implies that, for $i=1,2$, we can embed all trees from $A_i$ into $S_i\cup S_3$, thus  concluding the proof of  Lemma \ref{maya2}.

 Indeed, if both $A_1$ and $A_2$ contain elements of $B$, then by Claim~\ref{partition_of_A}~\eqref{los_panchosX}, they contain roughly the same number of vertices. By~\eqref{smallB}, each $A_i$ has few enough components from $B$ to ensure that there is a reasonable number of vertices in components which can be embedded with at least two fifths in $S_3$. So we can embed all trees from $A_i$, leaving at most $15\sqrt\gamma_0 m$ vertices from $S_i$ unused (here we also use~\eqref{minnndeg},~\eqref{minndeg}, and~\eqref{Siiii}).  
 
 On the other hand, if only the smaller set among $A_1$ and $A_2$ contains elements of $B$, then we can embed this set as before. For the other set we recall that since it does not contain any small trees, half of its vertices (or less, if desired) can be embedded into $S_3$. So we finish the embedding without a problem.

\newcommand{\etalchar}[1]{$^{#1}$}

\end{document}